\documentclass[11pt]{amsart}
\usepackage{amsfonts,amsthm,amsmath}
\usepackage{natbib}
\usepackage{hyperref}

\newtheorem{theorem}{Theorem}[section]
\newtheorem{lemma}{Lemma}[section]

\newcommand{\nchoosek}{\left(\begin{array}cn\\k\end{array}\right)}
\newcommand{\nchoosel}{\left(\begin{array}cn\\l\end{array}\right)}

\newcommand{\one}{{\bf 1}}
\newcommand{\reals}{{\mathbb R}}
\newcommand{\bbr}{\reals}
\newcommand{\vep}{\varepsilon}

\def\Var{{\rm Var}}

\numberwithin{equation}{section}

\begin{document}

\title[Truncated heavy tails]{Large deviations for truncated heavy-tailed random variables: a boundary case}
\author[A. Chakrabarty]{Arijit Chakrabarty}
\address{Theoretical Statistics and Mathematics Unit,
Indian Statistical Institute,
203 B.T. Road,
Kolkata - 700108, India}
\email{arijit.isi@gmail.com}

\begin{abstract}
This paper investigates the decay rate of the probability that the row sum of a triangular array of truncated heavy tailed random variables is larger than an integer $(k)$ times the truncating threshold, as both - the number of summands and the threshold go to infinity. The method of attack for this problem is significantly different from the one where $k$ is not an integer, and requires much sharper estimates.
\end{abstract}

 \keywords{ heavy tails, truncation, regular variation,  large deviation\vspace{.5ex}}
%\thanks{Research partly supported by a NBHM postdoctoral fellowship at the Indian Statistical Institute.}

\maketitle

\section{Introduction}
There are systems where the occurrence of a particularly large value is more likely than usual. Often, such systems are modeled by a distribution whose tail is regularly varying near infinity. At the same time, it may be the case that there is a natural upper bound on the set of possible values. This makes imperative the analysis of so called truncated regularly varying distributions. 

Suppose $Y, Y_1, Y_2,\ldots$ are i.i.d.\ random variables with $P(Y>\cdot)$ regularly varying with index $-\alpha$, for some $\alpha>0$, and let $(M_n)$ be a sequence going to infinity.
%$$
%\lim_{n\to\infty}nP(Y>M_n)=0\,.
%$$
Define the triangular array $\{X_{nj}:1\le j\le n\}$ by
\begin{equation}\label{eq.model}
X_{nj}:=Y_j\one(|Y_j|\le M_n)\,,
\end{equation}
and denote the row sum by
\begin{equation}\label{eq.rowsum}
S_n:=\sum_{j=1}^nX_{nj}\,.
\end{equation}

The question of how fast should $M_n$ grow so that the asymptotic behaviour of $S_n$ resembles that of the partial sum of i.i.d.\ random variables from an untruncated regularly varying distribution has been looked at from various angles; see \cite{chakrabarty:samorodnitsky:2012} and \cite{chakrabarty:2012}. The latter answers the above question from the point of view of the large deviation probability 
\begin{equation}\label{intro.eq1}
P\left(S_n>kM_n\right)
\end{equation}
as $n\to\infty$, for a fixed $k>0$. It is established, almost beyond doubt, that when $M_n$ grows fast enough so that
\begin{equation}
\label{new.eq1}\lim_{n\to\infty}nP(Y>M_n)=0\,,
\end{equation}
the underlying truncated distribution resembles the corresponding untruncated one. Furthermore, when \eqref{new.eq1} holds, it turns out that the case when $k$ is an integer is more challenging than the case when $k$ is a fraction because of the following simple reason. When $k$ is not an integer, a result similar to Theorem 2.2 in \cite{chakrabarty:2012} holds, which essentially shows that the probability in \eqref{intro.eq1} is asymptotically equivalent to
$$
\left(\begin{array}cn\\\lceil k\rceil\end{array}\right)P\left(\sum_{j=1}^{\lceil k\rceil}X_{nj}>kM_n\right)\,,
$$
and therefore argues that the probability is of the order $\{nP(Y>M_n)\}^{\lceil k\rceil}$. When $k$ is an integer, the natural guess is that the same equivalence holds, with $\lceil k\rceil$ replaced by $k+1$, that is, $S_n$ will be larger than $kM_n$, ``if and only if'', $\sum_{i=1}^{k+1}X_{nj_i}$ is larger than $kM_n$ for some $1\le j_1<\ldots<j_{k+1}\le n$. However, as is shown in Theorem \ref{t1} below, this is not always the case. This disagreement with intuition is the primary reason behind the boundary case (that is the case when $k$ is an integer) being harder. The author could solve this boundary problem only when $\alpha\ge k$ or $\alpha<k/(k+2)$. Surprisingly, it turns out that the asymptotic magnitude of the probability is different for these two regimes of $\alpha$, a phenomenon one wouldn't guess a priori. The behavior in these two regimes, under the assumption \eqref{new.eq1}, are studied in Sections \ref{sec:alphagek} and \ref{sec:alphalk} respectively. In Section \ref{sec:examples}, we consider a couple of examples.

We conclude this section by pointing out some classical works in the theory of large deviations for (untruncated) heavy tailed random variables. The study started in late sixties with \cite{heyde:1968}, \cite{nagaev:1969b} and \cite{nagaev:1969a}, significant contributions being made later by \cite{nagaev:1979} and \cite{cline:hsing:1991}, among others. Recently, \cite{hult:lindskog:mikosch:samorodnitsky:2005} studied the functional version of the large deviations principle.

\section{The case $\alpha\ge k$}\label{sec:alphagek} Suppose that $Y,Y_1, Y_2,\ldots$ are i.i.d. random variables such that $P(|Y|>\cdot)$ is regularly varying with index $-\alpha$, for some $\alpha>1$, and
\begin{equation}\label{defp}
p:=\lim_{x\to\infty}\frac{P(Y>x)}{P(|Y|>x)}\mbox{ exists, and is positive}\,.
\end{equation}
We assume that
\begin{equation}\label{eq.zeromean}
E(Y)=0\,.
\end{equation}
If $\alpha=2$, we assume that
\begin{equation}\label{eq.fv}
E(Y^2)<\infty\,.
\end{equation}
Furthermore, we assume that given $\delta>0$, there exist $T_0>0$ and $u_0\in(0,1)$ such that for all $T\ge T_0$ and $1-u_0\le a\le b\le1$,
\begin{equation}\label{assume}
\left|\frac{P(Y>aT)-P(Y>bT)}{P(Y>T)}-(a^{-\alpha}-b^{-\alpha})\right|\le\delta(b-a)\,.
\end{equation}
In Section \ref{sec:examples}, it will be shown that if the function $x^\alpha P(Y>x)$ belongs to the Zygmund class, then the above is satisfied.

Define the quantile sequence $(b_n)$ as
\begin{equation}\label{eq.defbn}
b_n:=\inf\left\{x>0:P(Y>x)\le\frac1n\right\}\,.
\end{equation}
Let $(M_n)$ be a sequence of positive numbers so that
\begin{equation}\label{eq.hard:alphal2}
M_n\gg b_n,\mbox{ if }\alpha<2\,,
\end{equation}
and
\begin{equation}\label{eq.hard:alphage2}
M_n\gg n^{1/2+\gamma}\mbox{ for some }\gamma>0,\mbox{ if }\alpha\ge2\,,
\end{equation}
where $u_n\gg v_n$ (or $v_n\ll u_n$), for positive sequences $u_n$ and $v_n$, means that
\[
\lim_{n\to\infty}\frac{u_n}{v_n}=\infty\,,
\]
throughout the paper.
The triangular array $\{X_{nj}:1\le j\le n\}$ and their row sum $S_n$ are as defined in \eqref{eq.model} and \eqref{eq.rowsum} respectively.

The first result of this paper, Theorem \ref{t1} below, describes the decay rate of $P(S_n>kM_n)$ when $k$ is a positive integer and $\alpha\ge k$, under some additional assumptions.
The result below refers to stable distributions; an overview of this topic can be found in the first chapter of \cite{samorodnitsky:taqqu:1994}.

\begin{theorem}\label{t1} Fix an integer $k$ such that $1\le k\le\alpha$ (recall that $\alpha>1$).
%Suppose that $\alpha\ge k$, where $k$ is a positive integer (we still assume that $\alpha>1$). 
Assume furthermore that if $\alpha=k\ge3$,  then
\begin{equation}\label{l3.eq2}
\int_0^\infty y^\alpha P(Y\in dy)<\infty\,.
\end{equation}
Then
$$
P(S_n>kM_n)\sim c_n^kn^kM_n^{-k}P(Y>M_n)^k\frac{\alpha^k}{(k!)^2}\int_0^\infty s^kP(Z_\alpha\in ds)\,,
$$
as $n\to\infty$, where
$$
c_n:=\left\{\begin{array}{ll}b_n,&\alpha<2\,,\\n^{1/2},&\alpha\ge2\,,\end{array}\right.
$$
and $Z_\alpha$ follows an $\alpha$-stable distribution with scale, location and skewness parameters as $1$, $0$ and $2p-1$, respectively, with $p$ as in \eqref{defp}, if $\alpha<2$, and a normal distribution with mean zero and variance same as that of $Y$ if $\alpha\ge2$.
\end{theorem}

For proving the result, we need some lemmas. First, let us fix some notations. For $l\ge1$, let $C_{nl}$ be the set of $l$-tuples $j=(j_1,\ldots,j_l)$ such that $1\le j_1<\ldots<j_l\le n$. For any $j\in C_{nl}$, denote
$$
j^c:=\{1,\ldots,n\}\setminus\{j_1,\ldots,j_l\}\,.
$$

\begin{lemma}\label{l1}
Suppose that $(x_n)$ is a sequence satisfying
\begin{equation}\label{l1.eq1}
M_n\gg x_n\gg b_n,\mbox{ if }\alpha<2\,,
\end{equation}
and
\begin{equation}\label{l1.eq2}
M_n\gg x_n\gg n^{1/2+\gamma},\mbox{ if }\alpha\ge2\,,
\end{equation}
where $\gamma$ is same as that in \eqref{eq.hard:alphage2}. Then,
$$
P\left(\bigcap_{j\in C_{nk}}\left\{\sum_{i\in j^c}X_{ni}>x_n,S_n>kM_n\right\}\right)=O\left(n^{k+1}P(Y>x_n)^{k+1}\right)\,,
$$
for all fixed $k\ge1$, as $n\to\infty$.
\end{lemma}

\begin{proof}%[Proof of Lemma \ref{l1}]
The proof is by induction on $k$. The proof for $k=1$ is very similar to the induction step, and hence we do not show the former separately. As the induction hypothesis, assume that the result is true for $1,\ldots,k-1$, and we shall show it for $k$. Observe that if $\alpha\ge2$, then
$$
b_n=O\left(n^{1/2+\gamma}\right)\,,
$$
and hence by \eqref{l1.eq1} and \eqref{l1.eq2}, it follows that
\begin{equation}\label{l1.eq3}
b_n\ll x_n,\,
\end{equation}
for all $\alpha$. Since, $p$ as defined in \eqref{defp}, is positive, it follows that $P(Y>\cdot)$ is regularly varying with index $-\alpha$.
It immediately follows from \eqref{eq.defbn} that
\begin{equation}\label{eq1}
\lim_{n\to\infty}nP(Y>b_n)=1\,.
\end{equation}
Fix
$$
\vep\in\left(0,\frac1{k+2}\right)\,,
$$
and let $u$ be such that
$$
0<u<\frac1\alpha\left(\frac1{k+2}-\vep\right)\,,
$$
and
\begin{equation}\label{l1.eq.new}
\frac{1-u}{2-\alpha u}<\frac12+\gamma\,.
\end{equation}
Notice that by \eqref{eq1},
$$
n^{-u}b_n^{1/(k+2)-\vep}\sim P(Y>b_n)^ub_n^{1/(k+2)-\vep}\ll P(Y>x_n)^ux_n^{1/(k+2)-\vep}\,,
$$
the inequality following from \eqref{l1.eq3} and the fact that the function $x$ going to $P(Y>x)^ux^{1/(k+2)-\vep}$, is regularly varying with a positive index, namely $1/(k+2)-\vep-u\alpha$. Thus,
\begin{equation}\label{l1.eq4}
b_n^{1/(k+2)-\vep}x_n^{(k+1)/(k+2)+\vep}\ll n^uP(Y>x_n)^ux_n\,.
\end{equation}
Let $\delta\in(0,2-\alpha u)$ be such that
$$
\frac{1-u}{2-\alpha u-\delta}\le\frac12+\gamma\,;
$$
such a $\delta$ exists because of \eqref{l1.eq.new}.
When $\alpha\ge2$, observe that
\begin{eqnarray*}
x_n^2P(Y>x_n)^u \gg x_n^{2-\alpha u-\delta} \gg n^{(1/2+\gamma)(2-\alpha u-\delta)}\ge n^{1-u}\,.
\end{eqnarray*}
In view of \eqref{l1.eq4} and the above inequality, there exists a sequence $(z_n)$ satisfying
\begin{equation}\label{l1.eq5}
b_n^{\frac1{k+2}-\vep}x_n^{\frac{k+1}{k+2}+\vep}\ll z_n\ll n^uP(Y>x_n)^ux_n,\,\mbox{if }\alpha<2\,,
\end{equation}
and
\begin{equation}\label{l1.eq6}
b_n^{\frac1{k+2}-\vep}x_n^{\frac{k+1}{k+2}+\vep}+\frac n{x_n}\ll z_n\ll n^uP(Y>x_n)^ux_n,\,\mbox{if }\alpha\ge2\,.
\end{equation}
Fix such a $(z_n)$.

For a set $A$ and $m\ge1$, let $S_m(A)$ denote the family of all subsets of $A$ that have cardinality $m$. Fix $1\le l\le k-1$. For $j\in C_{nl}\,,$ define
$$
D_j:=\bigcap_{i\in S_{n-k}(j^c)}\left\{\sum_{u=1}^{n-k}X_{ni_u}>x_n,\sum_{v\in j^c}X_{nv}>(k-l)M_n\right\}\,;
$$
recall the definitions of $C_{nl}$ and $j^c$ from the text preceding the statement of the current lemma.
Define the events
\begin{eqnarray*}
E_n&:=&\left\{|X_{nj}|>z_n\mbox{ for at least }(k+2)\mbox{ many }j\mbox{'s }\le n\right\}\,,\\
F_n&:=&\left\{\sum_{j=1}^nX_{nj}\one(|X_{nj}|\le z_n)>M_n-\frac{x_n}2\right\}\,,\\
G_n&:=&\left\{X_{nj}>\frac{x_n}{2(k+1)}\mbox{ for at least }(k+1)\mbox{ many }j\mbox{'s }\le n\right\}\,,\\
H_n&:=&\bigcup_{j\in C_{nk}}\Biggl\{X_{nj_i}>\frac{x_n}{2(k+1)}\mbox{ for }1\le i\le k\\
&&\mbox{ and }\sum_{i\in j^c}X_{ni}\one(|X_{ni}|\le z_n)>\frac{x_n}2\Biggr\}\,,\\
I_n&:=&\bigcup_{l=1}^{k-1}\bigcup_{j\in C_{nl}}\left[\left\{X_{nj_i}>\frac{x_n}{2(k+1)}\mbox{ for }1\le i\le l\right\}\cap D_j\right]\,.
\end{eqnarray*}
Our claim is that
\begin{equation}\label{l5.inclusion}
\bigcap_{j\in C_{nk}}\left\{\sum_{i\in j^c}X_{ni}>x_n,S_n>kM_n\right\}\subset E_n\cup F_n\cup G_n\cup H_n\cup I_n\,.
\end{equation}
To see this, if possible, fix a sample point in the left hand side, which is in neither of $E_n$, $F_n$, $G_n$, $H_n$ or $I_n$. Let
$$
l:=\#\left\{u\in\{1,\ldots,n\}:X_{nu}>\frac{x_n}{2(k+1)}\right\}\,.
$$
{\bf Case 1: $l=0$.} Write
$$
S_n:=\sum_{j=1}^nX_{nj}\one(|X_{nj}|\le z_n)+\sum_{j=1}^nX_{nj}\one(|X_{nj}|>z_n)\,.
$$
By the assumption that the sample point does not belong to $F_n$, it follows that the first sum on the right is at most $M_n-x_n/2$.
Since $E_n$ does not hold, the number of surviving summands in the second sum is at most $(k+1)$. As $l=0$, it follows that each summand is at most $x_n/2(k+1)$. Therefore, the second sum is at most $x_n/2$. Thus, $S_n\le M_n$, which is a contradiction.\\
{\bf Case 2: $1\le l\le k-1$.} Let $1\le j_1<\ldots<j_l\le n$ be such that
$$
X_{nj_i}>\frac{x_n}{2(k+1)}\mbox{ for all }1\le i\le l\,.
$$
Denote $j:=(j_1,\ldots,j_l)\in C_{nl}$. Clearly $D_j$ is a superset of the left hand side of \eqref{l5.inclusion}, and hence trivially the sample point is in $D_j$. Thus, the sample point is in $I_n$, which is a contradiction.\\
{\bf Case 3: $l=k$.} Once again, let $j_1<\ldots<j_k$ denote the indices $i$ for which $X_{ni}>x_n/2(k+1)$, and set $j:=(j_1,\ldots,j_k)$. Write
$$
\sum_{i\in j^c}X_{ni}=\sum_{i\in j^c}X_{ni}\one(|X_{ni}|\le z_n)+\sum_{i\in j^c}X_{ni}\one(|X_{ni}|>z_n)\,.
$$
Since $H_n$ does not hold, the first sum on the right is at most $x_n/2$. As $E_n$ does not hold, in the second sum, at most $(k+1)$ terms survive. Also, each summand in that sum is at most $x_n/2(k+1)$. Thus, the second sum is at most $x_n/2$. This shows that the left hand side is at most $x_n$, which clearly is a contradiction.\\
{\bf Case 4: $l\ge k+1$.} This case cannot arise because $G_n$ does not hold. Thus, the inclusion \eqref{l5.inclusion} is true.

In view of \eqref{l5.inclusion}, all that needs to be shown is that
\begin{equation}\label{l1.eq7}
P(E_n)+P(F_n)+P(G_n)+P(H_n)+P(I_n)=O\left[\left\{nP(Y>x_n)\right\}^{k+1}\right]\,.
\end{equation}
To that end, notice that
\begin{eqnarray}
P(E_n)&\le&n^{k+2}P\left(|Y|>z_n\right)^{k+2}\nonumber\\
&=&O\left[n^{k+2}P\left(Y>z_n\right)^{k+2}\right]\,,\label{l1.eq9}
\end{eqnarray}
the second step following from the assumption that $p$, as defined in \eqref{defp}, is positive.
By \eqref{l1.eq3} and \eqref{eq1}, it follows that
$$
b_n^{1/(k+2)-\vep}x_n^{(k+1)/(k+2)+\vep}\gg b_n\,,
$$
and
\begin{equation}\label{l1.eq10}
\lim_{n\to\infty}nP(Y>x_n)=0\,.
\end{equation}
These, in view of \eqref{l1.eq5} and \eqref{l1.eq6}, show that
\begin{equation}\label{l1.eq8}
b_n\ll z_n\ll x_n\,.
\end{equation}
Set
$$
\theta:=\alpha(k+2)\vep\,,
$$
and observe that as $n\to\infty$,
\begin{eqnarray*}
\frac{n^{k+2}P(Y>z_n)^{k+2}}{n^{k+1}P(Y>x_n)^{k+1}}&\sim&\left\{\frac{P(Y>z_n)}{P(Y>x_n)}\right\}^{k+1}\frac{P(Y>z_n)}{P(Y>b_n)}\\
&\le&2\left\{\left(\frac{z_n}{x_n}\right)^{-\alpha-\theta/(k+1)}\right\}^{k+1}\left(\frac {z_n}{b_n}\right)^{-\alpha+\theta}\\
&=&2\left(\frac{z_n}{b_n^{1/(k+2)-\vep}x_n^{(k+1)/(k+2)+\vep}}\right)^{-\alpha(k+2)}\\
&\to&0\,,
\end{eqnarray*}
the inequality in the second line holding for large $n$, and following by the Potter bounds (Proposition 2.6 in \cite{resnick:2007}) and \eqref{l1.eq8}. This, in view of \eqref{l1.eq9}, shows that
$$
P(E_n)=o\left(n^{k+1}P(Y>x_n)^{k+1}\right)\,.
$$
It is obvious that
$$
P(G_n)=O\left(n^{k+1}P(Y>x_n)^{k+1}\right)\,.
$$
Next we proceed to show that
\begin{equation}\label{l1.eq11}
P(H_n)=O\left(n^{k+1}P(Y>x_n)^{k+1}\right)\,.
\end{equation}
To that end, we shall use a result from \cite{prokhorov:1959}, which states that if $T_1, T_2,\ldots,T_N$ are independent zero mean random variables such that $|T_j|\le C$ for all $1\le j\le n$, then
\begin{equation}\label{l1.prokhorov}
P\left(\sum_{j=1}^NT_j>\lambda\right)\le\exp\left\{-\frac\lambda{2C}\sinh^{-1}\frac{\lambda C}{2\Var(\sum_{j=1}^NT_j)}\right\}\,.
\end{equation}
By the above, it follows that
\begin{eqnarray*}
&&P(H_n)\\
&\le&n^kP\left(Y>\frac{x_n}{2(k+1)}\right)^kP\left(\sum_{j=k+1}^nX_{nj}\one(|X_{nj}|\le z_n)>\frac12x_n\right)\\
&=&O\left[P\left(\sum_{j=k+1}^nX_{nj}\one(|X_{nj}|\le z_n)>\frac12x_n\right)\right]\,,
\end{eqnarray*}
the last step following from \eqref{l1.eq10}. The assumption \eqref{eq.zeromean} implies that for $n$ large,
\begin{eqnarray}
nE\left[X_{n1}\one(|X_{n1}|\le z_n)\right]&=&-nE\left[Y\one(|Y|>z_n)\right]\nonumber\\
&=&O\left(nz_nP(Y>z_n)\right)\nonumber\\
&=&o\left(nx_nP(Y>x_n)\right)\nonumber\\
&=&o(x_n)\,,\label{l1.eq17}
\end{eqnarray}
where the second step follows from Karamata's theorem; see Theorem 2.1 in \cite{resnick:2007}.
Thus, for $n$ large enough,
\begin{eqnarray}
&&P\left(\sum_{j=k+1}^nX_{nj}\one(|X_{nj}|\le z_n)>\frac12x_n\right)\nonumber\\
&\le& P\left(\sum_{j=k+1}^nX_{nj}\one(|X_{nj}|\le z_n)-(n-k)E\left[X_{n1}\one(|X_{n1}|\le z_n)\right]>\frac14x_n\right)\nonumber
\end{eqnarray}
\begin{equation}
\le\exp\left\{-K_1\frac{x_n}{z_n}\sinh^{-1}K_2\frac{z_nx_n}{(n-k)\Var(X_{n1}\one(|X_{n1}|\le z_n))}\right\}\,,\label{l1.eq15}
\end{equation}
for some finite positive constants $K_1$ and $K_2$. Thus, for \eqref{l1.eq11}, it suffices to show that
\begin{eqnarray}
x_nz_n&\gg&n\Var(X_{n1}\one(|X_{n1}|\le z_n))\,,\label{l1.eq13}\\
\mbox{and }\frac{x_n}{z_n}&\gg&\left(nP(Y>x_n)\right)^{-u}\,.\label{l1.eq14}
\end{eqnarray}
This is because if the above hold, then by \eqref{l1.eq15} and the fact that
$$
\lim_{x\to\infty}\sinh^{-1}x=\infty\,,
$$
it follows that,
$$
P(F_n)=O\left(\exp\left\{-K_1(nP(Y>x_n))^{-u}\right\}\right)=o\left[(nP(Y>x_n))^{k+1}\right]\,.
$$
Notice that \eqref{l1.eq14} is a restatement of the second inequalities in \eqref{l1.eq5} and \eqref{l1.eq6}. We shall show \eqref{l1.eq13} separately for the cases $\alpha<2$ and $\alpha\ge2$.\\
{\bf Case $\alpha<2$:} By the Karamata's theorem, it follows that
\begin{eqnarray*}
n\Var(X_{n1}\one(|X_{n1}|\le z_n))&=&O\left(nz_n^2P(Y>z_n)\right)\\
&=&o(z_n^2)\\
&=&o(z_nx_n)\,.
\end{eqnarray*}
{\bf Case $\alpha\ge2$:} The assumption \eqref{eq.fv} implies that
\begin{eqnarray*}
n\Var(X_{n1}\one(|X_{n1}|\le z_n))&=&O(n)\\
&=&o(x_nz_n)\,,
\end{eqnarray*}
the last step following from the left inequality in \eqref{l1.eq6}. Thus, \eqref{l1.eq13} holds, and hence so does \eqref{l1.eq11}.

By similar arguments as above, the fact that
\begin{equation}\label{l1.eq16}
P(F_n)=O\left(n^{k+1}P(Y>x_n)^{k+1}\right)
\end{equation}
will follow once it can be shown that
$$
nE\left[X_{n1}\one(|X_{n1}|\le z_n)\right]=o(\hat M_n)\,,
$$
$$
\hat M_nz_n\gg n\Var(X_{n1}\one(|X_{n1}|\le z_n))\,,
$$
and
$$
\frac{\hat M_n}{z_n}\gg \left(nP(Y>x_n)\right)^{-u}\,,
$$
where
$$
\hat M_n:=M_n-\frac12x_n\,.
$$
These, follow immediately from \eqref{l1.eq17}, \eqref{l1.eq13} and \eqref{l1.eq14} respectively, along with the fact that $x_n\ll\hat M_n$. This establishes \eqref{l1.eq16}.

To show \eqref{l1.eq7}, and thus complete the proof of the lemma, all that needs to be shown is
\begin{equation}\label{l1.eq18}
P(I_n)=O\left(n^{k+1}P(Y>x_n)^{k+1}\right)\,.
\end{equation}
This is the only place where we shall use the induction hypothesis; note that when $k=1$, $I_n$ is the null event, and hence the above arguments give a complete proof for that case. To the end of proving this for general $k$, observe that
\begin{eqnarray*}
P(I_n)&\le&\sum_{l=1}^{k-1}\sum_{j\in C_{nl}}P\left(\left\{X_{nj_i}>\frac{x_n}{2(k+1)}\mbox{ for }1\le i\le l\right\}\cap D_j\right)\\
&\le&\sum_{l=1}^{k-1}\nchoosel P\left(Y>\frac{x_n}{2(k+1)}\right)^l P\left(D_{(1,\ldots,l)}\right)\\
&=&\sum_{l=1}^{k-1}O\left[n^lP\left(Y>{x_n}\right)^lP\left(D_{(1,\ldots,l)}\right)\right]\,.
\end{eqnarray*}
Notice that,
$$
P\left(D_{(1,\ldots,l)}\right)=O\left(n^{k-l+1}P(Y>x_n)^{k-l+1}\right)
$$
for all fixed $l=1,\ldots,k-1$, since by the induction hypothesis, the claim of the theorem is true for $k$ replaced by $k-l$.  This shows \eqref{l1.eq18}, and consequently the result for $k$, that is, the induction step, which in turn completes the proof.
\end{proof}

\begin{lemma}\label{l2}
For $k\ge1$, there exists a function $\vep_k:(0,1)\times(0,\infty)\to\bbr$ such that
\begin{equation}\label{l2.claim}
\lim_{T\to\infty,u\downarrow0}u^{-k}\vep_k(u,T)=0\,,
\end{equation}
and
\begin{equation}\label{l2.eq3}
\frac{P\left(\sum_{i=1}^kX_{ni}>(k-u)M_n\right)}{P(Y>M_n)^k}=\frac{\alpha^k}{k!}u^k+\vep_k(u,M_n)\,,
\end{equation}
for all $n\ge1$, $0<u<1$.
\end{lemma}

\begin{proof}
The proof is by induction on $k$. First let us show the result for $k=1$. To that end, define
$$
\vep_1(u,T):=\frac{P(Y>(1-u)T)-P(Y>T)}{P(Y>T)}-\alpha u\,,
$$
and notice that \eqref{l2.eq3} holds with this $\vep_1(\cdot,\cdot)$. Also,
$$
|\vep_1(u,T)|\le\left|\frac{P(Y>(1-u)T)-P(Y>T)}{P(Y>T)}-\left\{(1-u)^{-\alpha}-1\right\}\right|
$$
$$
\,\,\,\,\,+\left|\left\{(1-u)^{-\alpha}-1\right\}-\alpha u\right|\,.
$$
By \eqref{assume}, the first term on the right hand side is $o(u)$ as $T\to\infty$ and $u\downarrow0$. The second term is $o(u)$ by Taylor's theorem. This shows the result for $k=1$. Suppose now that the claim is true for $k$; we shall show the same for $k+1$. Clearly, for $n\ge1$ and $0<u<1$,
\begin{eqnarray*}
&&P\left(\sum_{i=1}^{k+1}X_{ni}>(k+1-u)M_n\right)\\
&=&\int_{(1-u,1]}P\left(\sum_{i=1}^kX_{ni}>\left\{k-(u+x-1)\right\}M_n\right)P\left(M_n^{-1}Y\in dx\right)\\
&=&P(Y>M_n)^k\int_{(1-u,1]}\left[\frac{\alpha^k}{k!}(u+x-1)^k+\vep_k(u+x-1,M_n)\right]\\
&&\,\,\,\,\,\,\,\,\,\,\,\,\,\,\,\,\,\,\,\,\,\,\,\,\,\,\,\,\,\,\,\,\,\,\,\,\,\,\,\,\,\,\,\,\,\,\,\,\,\,\,\,\,\,\,\,\,\,\,\,\,\,\,\,\,\,\,\,\,\,\,\,P\left(M_n^{-1}Y\in dx\right)\,.
\end{eqnarray*}
Define
$$
\vep_{k+1}(u,T):=
$$
$$
\int_{(1-u,1]}\left[\frac{\alpha^k}{k!}(u+x-1)^k+\vep_k(u+x-1,T)\right]\frac{P\left(T^{-1}Y\in dx\right)}{P(Y>T)}-\frac{\alpha^{k+1}}{(k+1)!}u^{k+1}\,.
$$
To complete the proof, all that needs to be shown is
\begin{equation}\label{l2.eq1}
\lim_{T\to\infty,u\downarrow0}u^{-(k+1)}\vep_{k+1}(u,T)=0\,.
\end{equation}
To that end, we shall first show
\begin{equation}\label{l2.eq2}
\lim_{T\to\infty,u\downarrow0}u^{-(k+1)}\int_{(1-u,1]}(u+x-1)^k\frac{P\left(T^{-1}Y\in dx\right)}{P(Y>T)}=\frac\alpha{k+1}\,.
\end{equation}
Fix $\alpha>\delta>0$, and $T_0>0$, $u_0\in(0,1)$ such that \eqref{assume} holds. Then, for $T\ge T_0$ and
$1-u_0\le a<b\le1$,
$$
\alpha-\delta\le\frac{P(Y\in(aT,bT])}{(b-a)P(Y>T)}\le\alpha(1-u_0)^{-\alpha-1}+\delta\,.
$$
Thus, for $T\ge T_0$ and $0<u\le u_0$,
$$
(\alpha-\delta)\int_{1-u}^1(u+x-1)^kdx\le\int_{(1-u,1]}(u+x-1)^k\frac{P\left(T^{-1}Y\in dx\right)}{P(Y>T)}
$$
\begin{equation}\label{l2.eq4}
\le\left\{\alpha(1-u_0)^{-\alpha-1}+\delta\right\}\int_{1-u}^1(u+x-1)^kdx\,.
\end{equation}
Since $\delta$ and $u_0$ can be chosen to be arbitrarily small, it follows that as $T\to\infty$ and $u\downarrow0$,
\begin{eqnarray*}
\int_{(1-u,1]}(u+x-1)^k\frac{P\left(T^{-1}Y\in dx\right)}{P(Y>T)}&\sim&\alpha\int_{1-u}^1(u+x-1)^kdx\\
&=&\frac\alpha{k+1}u^{k+1}\,.
\end{eqnarray*}
This shows \eqref{l2.eq2}. Now using the induction hypothesis that
$$
\lim_{T\to\infty,u\downarrow0}u^{-k}\vep_k(u,T)=0\,,
$$
and the above computations, \eqref{l2.eq1} follows. This establishes the induction step, and thus completes the proof of the lemma.
\end{proof}

\begin{lemma}\label{l3}
Under the assumptions of Theorem \ref{t1},
\begin{equation}\label{l3.eq1}
\lim_{n\to\infty}\int_0^\infty s^kP\left(c_n^{-1}S_n\in ds\right)=\int_0^\infty s^kP(Z_\alpha\in ds)<\infty\,,
\end{equation}
 where $(c_n)$ and $Z_\alpha$ are as defined in the statement of that theorem.
\end{lemma}

\begin{proof}%[Proof of Lemma \ref{l3}]
Since $\alpha$ is always assumed to be larger than $1$, notice that either $k<\alpha<2$ or $\alpha\ge2$ always holds. In the former case, $Z_\alpha$ has an $\alpha$-stable distribution, and hence the integral on the right hand side of \eqref{l3.eq1} is finite because $k<\alpha$. In the latter case $Z_\alpha$ is a normal random variable, and hence the integral is finite for all $k$.

Notice that the assumption \eqref{eq.zeromean} implies that as $n\to\infty$,
$$
c_n^{-1}\sum_{j=1}^nY_j\Longrightarrow Z_\alpha\,.
$$
Recall that $S_n$, defined in \eqref{eq.rowsum}, is the row sum of the triangular array $\{X_{nj}:1\le j\le n\}$.
Notice that
\begin{eqnarray*}
P\left(S_n\neq\sum_{j=1}^nY_j\right)&\le&nP(|Y|>M_n)\\
&=&o(1)\,,
\end{eqnarray*}
the last equality being an immediate consequence of \eqref{eq.hard:alphal2} and \eqref{eq.hard:alphage2}. Thus, it follows that
\begin{equation}\label{eq.conv}
c_n^{-1}S_n\Longrightarrow Z_\alpha\,.
\end{equation}

For the proof, we shall use Theorem 3.2 in \cite{acosta:gine:1979}, which states the following:\\
Let $\{V_{nj}:1\le j\le n\}$ be a triangular array satisfying
\begin{equation}\label{l6.eq1}
\lim_{n\to\infty}\max_{1\le j\le n}P\left(|V_{nj}|>\vep\right)=0\mbox{ for all }\vep>0\,,
\end{equation}
and the row sum $\tilde S_n:=\sum_{j=1}^nV_{nj}$ converges weakly to a probability measure $\nu$. If $\phi$ is a continuous function from $\bbr$ to $[0,\infty)$ such that there exists $a\in(0,\infty)$ with
\begin{equation}\label{l6.eq2}
\phi(x+y)\le a\phi(x)\phi(y)\mbox{ for all }x,y\,,
\end{equation}
and
\begin{equation}\label{l6.eq3}
\lim_{T\to\infty}\sup_{n\ge1}\sum_{j=1}^nE\left[\phi(V_{nj})\one(|V_{nj}|>T)\right]=0\,,
\end{equation}
then
$$
\lim_{n\to\infty}E\phi(\tilde S_n)=\int_\bbr\phi d\nu\,.
$$
The plan is to use the above result with $\phi$ defined by
$$
\phi(x):=2+|x|^k\one(x>0)\,.
$$
A quick inspection will reveal that $\phi$ thus defined, satisfies \eqref{l6.eq2} with $a=2^k$. Let $\{X_{nj}\}$ be as defined in \eqref{eq.model}, and set
$$
V_{nj}:=c_n^{-1}X_{nj},\,1\le j\le n\,.
$$
Thus, checking \eqref{l6.eq1} and \eqref{l6.eq3} suffices for the proof of the lemma. Verifying \eqref{l6.eq1} is trivial. A quick argument will yield that instead of showing \eqref{l6.eq3}, it is enough to show
\begin{equation}\label{l3.eq6}
\lim_{T\to\infty}\limsup_{n\to\infty}\sum_{j=1}^nE\left[\phi(V_{nj})\one(|V_{nj}|>T)\right]=0\,.
\end{equation}
For \eqref{l3.eq6}, observe that
\begin{eqnarray*}
&&\sum_{j=1}^nE\left[\phi(V_{nj})\one(|V_{nj}|>T)\right]\\
&=&2nP\left(|X_{n1}|>c_nT\right)+nc_n^{-k}E\left[X_{n1}^k\one\left(X_{n1}>c_nT\right)\right]\\
&\le&2nP\left(|Y|>c_nT\right)+nc_n^{-k}\int_{[c_nT,\infty)} y^kP(Y\in dy)\,.
\end{eqnarray*}

By \eqref{eq.fv}, it follows that when $\alpha\ge2$,
\begin{equation}\label{eq.alpha2}
P(Y>x)=o(x^{-2})\mbox{ as }x\to\infty\,.
\end{equation}
This, along with \eqref{eq1}, shows that
\begin{equation}\label{l3.eq4}
\limsup_{n\to\infty}nP(|Y|>c_n)<\infty\,,
\end{equation}
for all $\alpha$. In view of this, it follows that there is a finite constant $C$, independent of $T$ and $n$, satisfying
$$
\limsup_{n\to\infty}nP\left(|Y|>c_nT\right)=CT^{-\alpha}\,.
$$
Thus, it follows that
\begin{equation}\label{l3.eq5}
\lim_{T\to\infty}\limsup_{n\to\infty}nP\left(|Y|>c_nT\right)=0\,.
\end{equation}
For \eqref{l3.eq6}, it remains to show that
\begin{equation}\label{l3.eq3}
\lim_{T\to\infty}\limsup_{n\to\infty}nc_n^{-k}\int_{[c_nT,\infty)} y^kP(Y\in dy)=0\,.
\end{equation}
We shall prove \eqref{l3.eq3} separately for the cases $\alpha>k$ and $\alpha=k$.\\
{\bf Case $\alpha>k$:} Observe that for fixed $T>0$, by the Karamata's theorem,
\begin{eqnarray*}
\limsup_{n\to\infty}nc_n^{-k}\int_{[c_nT,\infty)} y^kP(Y\in dy)&=&\frac \alpha{\alpha-k}T^k\limsup_{n\to\infty}nP(Y>c_nT)\\
&\le&C\frac \alpha{\alpha-k}T^{k-\alpha}\,.
\end{eqnarray*}
This completes the proof of \eqref{l3.eq3}.\\
{\bf Case $\alpha=k$:} Since the assumption $\alpha>1$ is in force, for this case it is necessarily true that $\alpha=k\ge2$. Thus, for $T\ge1$,
\begin{eqnarray*}
nc_n^{-k}\int_{[c_nT,\infty)} y^kP(Y\in dy)&\le&nc_n^{-k}\int_{[c_n,\infty)} y^kP(Y\in dy)\\
&=&n^{1-k/2}\int_{[n^{1/2},\infty)} y^\alpha P(Y\in dy)\\
&\le&\int_{[n^{1/2},\infty)} y^\alpha P(Y\in dy)\,.
\end{eqnarray*}
By the assumptions \eqref{eq.fv} and  \eqref{l3.eq2}, it follows that rightmost quantity goes to zero as $n\to\infty$. This shows \eqref{l3.eq3}. Equations \eqref{l3.eq5} and \eqref{l3.eq3} establish \eqref{l3.eq6}, which completes the proof.
\end{proof}

\begin{proof}[Proof of Theorem \ref{t1}]
We start with proving the upper bound, that is,
\begin{equation}\label{t1.ub}
\limsup_{n\to\infty}\frac{P(S_n>kM_n)}{n^kP(Y>M_n)^{k}c_n^kM_n^{-k}}\le\frac{\alpha^k}{(k!)^2}\int_0^\infty s^kP(Z_\alpha\in ds)\,.
\end{equation}
To that end, observe that
\begin{equation}\label{t1.eq2}
M_n\gg c_n\,,
\end{equation}
which is a consequence of \eqref{eq.hard:alphal2} and \eqref{eq.hard:alphage2}.
Our first claim is that
\begin{equation}\label{t1.eq1}
M_n^kP(Y>M_n)\ll c_n^kn^{-1}\,.
\end{equation}
We shall show this separately for the cases $\alpha>k$ and $\alpha=k$.\\
{\bf Case $\alpha>k$:}  By \eqref{t1.eq2} and the fact that $\alpha>k$, it follows that
\begin{eqnarray*}
M_n^kP(Y>M_n)\ll c_n^kP(Y>c_n)=O\left(c_n^k/n\right)\,,
\end{eqnarray*}
the rightmost equality following from \eqref{l3.eq4}.\\
{\bf Case $\alpha=k$:} It is necessary that in this case $\alpha\ge2$. By \eqref{eq.alpha2} and the assumption \eqref{l3.eq2}, it follows that
$$
M_n^kP(Y>M_n)\ll1\le n^{k/2-1}=c_n^kn^{-1}\,.
$$
This completes the proof of \eqref{t1.eq1}. Thus, it follows that
\begin{equation}\label{t1.eq6}
P(Y>M_n)^{k+1}\ll c_n^kn^{-1}M_n^{-k}P(Y>M_n)^k\,.
\end{equation}
Set
$$
\beta:=\left\{\begin{array}{ll}0,&\alpha<2\,,\\\gamma,&\alpha\ge2\,,\end{array}\right.
$$
where $\gamma$ is same as that in \eqref{eq.hard:alphage2}. It follows that
$$
M_n\gg n^\beta c_n\,.
$$
This, along with \eqref{t1.eq6}, shows that
$$
P(Y>M_n)^{k+1}\ll\min\left\{c_n^kn^{-1}M_n^{-k}P(Y>M_n)^k,P(Y>n^\beta c_n)^{k+1}\right\}\,.
$$
The right hand side, clearly, goes to zero. Thus, there exists a sequence $(T_n)$ satisfying
\begin{eqnarray}
P(Y>M_n)^{k+1}&\ll&P(Y>T_nc_n)^{k+1}\nonumber\\
&\ll&\min\left\{c_n^kn^{-1}M_n^{-k}P(Y>M_n)^k,P(Y>n^\beta c_n)^{k+1}\right\}\,.\label{t1.eq3}
\end{eqnarray}
Let us record some quick consequences of the above. One immediate observation is that
\begin{equation}\label{t1.eq4}
n^\beta c_n\ll c_nT_n\ll M_n\,.
\end{equation}
In particular,
\begin{equation}\label{t1.eq5}
\lim_{n\to\infty}T_n=\infty\,.
\end{equation}
Notice that
$$
\{S_n>kM_n\}\subset\left(\bigcup_{j\in C_{nk}}\left\{\sum_{u=1}^kX_{nj_u}>kM_n-T_nc_n,S_n>kM_n\right\}\right)
$$
$$
\cup\left(\bigcap_{j\in C_{nk}}\left\{\sum_{i\in j^c}X_{ni}>T_nc_n, S_n>kM_n\right\}\right)\,,
$$
where $C_{nk}$, as defined earlier, denotes the set of  $k$-tuples\\ $j=(j_1,\ldots,j_k)$ such that $1\le j_1<\ldots<j_k\le n$. Thus,
\begin{eqnarray*}
&&P(S_n>kM_n)\\
&\le&\nchoosek P\left(\sum_{i=1}^kX_{ni}>kM_n-T_nc_n,S_n>kM_n\right)+\\
&&\,\,\,\,P\left(\bigcap_{j\in C_{nk}}\left\{\sum_{i\in j^c}X_{ni}>T_nc_n, S_n>kM_n\right\}\right)\\
&=:& Q_1+Q_2\,.
\end{eqnarray*}
Clearly,
\begin{eqnarray*}
&&\nchoosek^{-1}Q_1\\
&=&\int_{(0,T_n]}P\left(\sum_{i=1}^kX_{ni}>kM_n-sc_n\right)P\left(c_n^{-1}\sum_{i=k+1}^nX_{ni}\in ds\right)\\
&&+P\left(\sum_{i=1}^kX_{ni}>kM_n-T_nc_n\right)P\left(\sum_{i=k+1}^nX_{ni}>T_nc_n\right)\\
&=:&Q_{11}+Q_{12}\,.
\end{eqnarray*}
Denote
\begin{equation}\label{t3.defPn}
P_n(ds):=P\left(c_n^{-1}\sum_{i=k+1}^nX_{ni}\in ds\right)\,.
\end{equation}
By Lemma \ref{l2}, it follows that
\begin{eqnarray*}
Q_{11}&=&P(Y>M_n)^{k}\frac{\alpha^k}{k!}c_n^kM_n^{-k}\int_{(0,T_n]}s^kP_n(ds)\\
&&\,\,\,\,\,\,\,\,\,\,\,\,\,\,+P(Y>M_n)^{k}\int_{(0,T_n]}\vep_k\left(sc_nM_n^{-1},M_n\right)P_n(ds)\\
&=:&Q_{111}+Q_{112}\,,
\end{eqnarray*}
where $\vep_k(\cdot,\cdot)$ satisfies \eqref{l2.claim}. Clearly,
\begin{eqnarray*}
Q_{111}&\le&P(Y>M_n)^{k}\frac{\alpha^k}{k!}c_n^kM_n^{-k}\int_0^\infty s^kP_n(ds)\\
&\sim&P(Y>M_n)^{k}\frac{\alpha^k}{k!}c_n^kM_n^{-k}\int_0^\infty s^kP(Z_\alpha\in ds)\,,
\end{eqnarray*}
the equivalence following from Lemma \ref{l3}. Also,
\begin{eqnarray*}
|Q_{112}|&\le&P(Y>M_n)^{k}c_n^kM_n^{-k}\int_0^{\infty}s^kP_n(ds) \sup_{0<t\le c_nT_nM_n^{-1}}\frac{|\vep_k(t,M_n)|}{t^k}\\
&=&o\left(P(Y>M_n)^{k}c_n^kM_n^{-k}\right)\,,
\end{eqnarray*}
the second equality following by \eqref{l2.claim}, \eqref{t1.eq4} and Lemma \ref{l3}. Hence,
\begin{equation}\label{t1.eq7}
\limsup_{n\to\infty}\frac{Q_{11}}{P(Y>M_n)^{k}c_n^kM_n^{-k}}\le\frac{\alpha^k}{k!}\int_0^\infty s^kP(Z_\alpha\in ds)\,.
\end{equation}
Next, we proceed to show that
\begin{equation}\label{t1.eq8}
Q_{12}=o\left(P(Y>M_n)^{k}c_n^kM_n^{-k}\right)\,.
\end{equation}
To that end, notice that
$$
P\left(\sum_{i=k+1}^nX_{ni}>T_nc_n\right)\le P\left(\sum_{i=k+1}^nY_i>T_nc_n\right)+nP(|Y|>M_n)\,.
$$
It is well known that
$$
P\left(\sum_{i=k+1}^nY_i>T_nc_n\right)=O\left(nP(Y>T_nc_n\right)\,;
$$
see, for example, Lemma 2.1 in \cite{hult:lindskog:mikosch:samorodnitsky:2005}. This, in view of \eqref{t1.eq4}, shows that
$$
P\left(\sum_{i=k+1}^nX_{ni}>T_nc_n\right)=O\left(nP(Y>T_nc_n\right)\,.
$$
Using Lemma \ref{l2} once again, it follows that
$$
P\left(\sum_{i=1}^kX_{ni}>kM_n-T_nc_n\right)=O\left(P(Y>M_n)^k(T_nc_nM_n^{-1})^k\right)\,.
$$
Thus,
\begin{eqnarray*}
Q_{12}&=&O\left(nP(Y>T_nc_n)P(Y>M_n)^k(T_nc_nM_n^{-1})^k\right)\,.
\end{eqnarray*}
In view of this, \eqref{t1.eq8} will follow if it can be shown that
\begin{equation}\label{t1.eq9}
\lim_{n\to\infty}nP(Y>T_nc_n)T_n^k=0\,.
\end{equation}
Once again, \eqref{t1.eq9} will be shown separately for the cases $\alpha>k$ and $\alpha=k$.\\
{\bf Case $\alpha>k$:} Fix $\delta\in(0,\alpha-k)$. By \eqref{l3.eq4}, it follows that there exists $K<\infty$, independent of $n$, such that
\begin{eqnarray*}
nP(Y>T_nc_n)T_n^k&\le&K\frac{P(Y>T_nc_n)}{P(Y>c_n)}T_n^k\\
&\le&2KT_n^{k-\alpha+\delta}\,,
\end{eqnarray*}
for $n$ large enough, the second inequality following by the Potter bounds and \eqref{t1.eq5}. This shows \eqref{t1.eq9}.\\
{\bf Case $\alpha=k$:} By \eqref{eq.alpha2}, and the assumption \eqref{l3.eq2}, it is ensured that
$$
nP(Y>T_nc_n)T_n^k\ll nc_n^{-k}=n^{1-k/2}\le1\,.
$$
Thus, \eqref{t1.eq9} holds, and hence so does \eqref{t1.eq8}. By \eqref{t1.eq7} and \eqref{t1.eq8}, it follows that
\begin{equation}\label{t1.eq10}
\limsup_{n\to\infty}\frac{Q_{1}}{n^kP(Y>M_n)^{k}c_n^kM_n^{-k}}\le\frac{\alpha^k}{(k!)^2}\int_0^\infty s^kP(Z_\alpha\in ds)\,.
\end{equation}
By Lemma \ref{l1} and \eqref{t1.eq4}, it follows that
\begin{eqnarray*}
Q_2&=&O\left(n^{k+1}P(Y>T_nc_n)^{k+1}\right)\\
&=&o\left(n^kP(Y>M_n)^{k}c_n^kM_n^{-k}\right)\,,
\end{eqnarray*}
the second inequality following from \eqref{t1.eq3}. This, along with \eqref{t1.eq10}, completes the proof of the upper bound \eqref{t1.ub}.

Next we proceed to establish the lower bound, that is
\begin{equation}\label{t1.lb}
\liminf_{n\to\infty}\frac{P(S_n>kM_n)}{n^kP(Y>M_n)^{k}c_n^kM_n^{-k}}\ge\frac{\alpha^k}{(k!)^2}\int_0^\infty s^kP(Z_\alpha\in ds)\,.
\end{equation}
To that end, fix $T>0$ and define the event
$$
B_j:=\left\{\sum_{i=1}^kX_{nj_i}>kM_n-Tc_n,S_n>kM_n\right\},\,j=(j_1,\ldots,j_k)\in C_{nk}\,.
$$
Clearly,
\begin{eqnarray*}
P(S_n>kM_n)&\ge&P\left(\bigcup_{j\in C_{nk}}B_j\right)\\
&\ge&\sum_{j\in C_{nk}}P(B_j)-{\sum}_{j^1,j^2\in C_{nk},j^1\neq j^2}P\left(B_{j^1}\cap B_{j^2}\right)\\
&=&\nchoosek P\left(B_{(1,\ldots,k)}\right)-\sum_{j^1,j^2\in C_{nk},j^1\neq j^2}P\left(B_{j^1}\cap B_{j^2}\right)\\
&=:&Q_3-Q_4\,.
\end{eqnarray*}
Note that
\begin{eqnarray*}
Q_3&\sim&\frac{n^k}{k!}P\left(B_{(1,\ldots,k)}\right)\\
&\ge&\frac{n^k}{k!}\int_{(0,T)}P\left(\sum_{i=1}^kX_{ni}>kM_n-sc_n\right)P_n(ds)\,,
\end{eqnarray*}
where $P_n(\cdot)$ is as defined in \eqref{t3.defPn}. By Lemma \ref{l2}, \eqref{eq.conv} and arguments similar to those leading to \eqref{t1.eq7}, it follows that as $n\to\infty$,
\begin{eqnarray*}
&&\int_{(0,T)}P\left(\sum_{i=1}^kX_{ni}>kM_n-sc_n\right)\\
&\sim&\frac{\alpha^k}{k!}M_n^{-k}P(Y>M_n)^kc_n^{k}\int_{(0,T)} s^kP(Z_\alpha\in ds)\,.
\end{eqnarray*}
This shows that
$$
\liminf_{n\to\infty}\frac{Q_3}{M_n^{-k}P(Y>M_n)^kc_n^{k}n^k}\ge\frac{\alpha^k}{(k!)^2}\int_{(0,T)} s^kP(Z_\alpha\in ds)\,.
$$
If it can be shown that
\begin{eqnarray}
Q_4=o\left(M_n^{-k}P(Y>M_n)^kc_n^{k}n^k\right)\,,\label{t3.eq15}
\end{eqnarray}
then \eqref{t1.lb} will follow by letting $T\to\infty$. To that end, write
$$
Q_4=\sum_{l=0}^{k-1}\sum_{j^1,j^2\in C_{nk}:\#(j^1\cap j^2)=l}P\left(B_{j^1}\cap B_{j^2}\right)\,.
$$
Fix $l\in\{0,\ldots,k-1\}$ and $j^1,j^2\in C_{nk}$ such that $\#(j^1\cap j^2)=l$.
Notice that for $n$ so large that $c_n^{-1}M_n\ge2T$,
\begin{eqnarray*}
&&P\left(B_{j^1}\cap B_{j^2}\right)\\
&\le&P\left(\sum_{i=1}^kX_{ni}>(k-1/2)M_n,\sum_{i=k-l+1}^{2k-l}X_{ni}>(k-1/2)M_n\right)\\
&\le&P\left(X_{ni}>M_n/2,\,1\le i\le 2k-l\right)\\
&\le&K_lP(Y>M_n)^{2k-l}\,,
\end{eqnarray*}
for some $K_l$ independent of $j^1$ and $j^2$. Thus,
\begin{eqnarray*}
Q_4&=&\sum_{l=0}^{k-1}O\left(n^{2k-l}P(Y>M_n)^{2k-l}\right)\\
&=&O\left(n^{k+1}P(Y>M_n)^{k+1}\right)\\
&=&o\left(M_n^{-k}P(Y>M_n)^kc_n^{k}n^k\right)\,,
\end{eqnarray*}
the equality in the last line following from \eqref{t1.eq1}.
This shows \eqref{t3.eq15}, which in turn establishes the lower bound \eqref{t1.lb}, and thus completes the proof.
\end{proof}

\section{The case $\alpha<k/(k+2)$}\label{sec:alphalk}
 Suppose that the random variables $Y_1,Y_2,\ldots$ are as defined in the beginning of Section \ref{sec:alphagek}, {\bf except that now we assume $0<\alpha<1$}. The assumption \eqref{assume} is still in force. Let $(M_n)$ be a real sequence, the assumption on which will be stated in the main result, namely Theorem \ref{t2} below. Suppose that  $\{X_{nj}:1\le j\le n\}$ and $S_n$ are as defined in \eqref{eq.model} and \eqref{eq.rowsum} respectively. The above mentioned result studies the behavior of $P(S_n>kM_n)$ when $k$ is a positive integer satisfying $\alpha<k/(k+2)$.
For stating the result, we need some more notations. For $k\ge1$, define the function $c_k$ from $[k-1,k]$ to $[0,\infty]$ recursively, by
\begin{eqnarray*}
c_1(t)&=&t^{-\alpha}-1\mbox{ for }0\le t\le1\,,\\
c_{k+1}(t)&=&\alpha\int_{t-k}^1 c_k(t-z)z^{-\alpha-1}dz\mbox{ for }k\le t\le k+1,\,k\ge1.
\end{eqnarray*}
In the preceding definition, we have followed the convention that $0^{-\alpha}=\infty$.

\begin{theorem}\label{t2} For a positive integer $k$,
\begin{equation}\label{t2.claim}
c_{k+1}(k)<\infty\,.
\end{equation}
Suppose that $0<\alpha<k/(k+2)$ and
\begin{equation}\label{t2.eq1}
M_n^{1-\frac{\alpha(k+2)}k-\gamma}\gg n^{\frac1\alpha}
\end{equation}
for some $0<\gamma<1-\frac{\alpha(k+2)}k$.
Then, as $n\to\infty$,
\begin{equation}\label{t2.claim2}
P(S_n>kM_n)\sim\frac{c_{k+1}(k)}{(k+1)!}n^{k+1}P(Y>M_n)^{k+1}\,.
\end{equation}
\end{theorem}

Notice that the assumption \eqref{t2.eq1} implies that
\begin{equation}\label{eq.observe}
\lim_{n\to\infty}nP(Y>M_n)=0\,,
\end{equation}
which in particular shows that this is stronger than the corresponding assumption \eqref{eq.hard:alphal2}, for the case $1<\alpha<2$ in Theorem \ref{t1}. Before starting the proof of Theorem \ref{t2}, we shall prove a couple of lemmas, which will be used in the proof of the former.

\begin{lemma}\label{l4}
For $k\ge1$, and $x\in(k-1,k]$,
\begin{equation}\label{l4.claim}
P\left(\sum_{j=1}^{k}X_{nj}>xM_n\right)\sim c_{k}(x)P(Y>M_n)^{k}\,,
\end{equation}
as $n\to\infty$, and \eqref{t2.claim} is true.
\end{lemma}

\begin{proof}%[Proof of Lemma \ref{l4}]
Since $k>\alpha$ by assumption, the proof of the lemma will be complete if we can show that for all $k\ge1$, \eqref{l4.claim} holds, and
\begin{equation}\label{l4.eq3}
c_{k}(k-u)=O\left(u^k\right)\,,
\end{equation}
as $u\downarrow0$.

The proof is by induction on $k$. For $k=1$, \eqref{l4.claim} and \eqref{l4.eq3} are trivial to check. As the induction hypothesis, we assume them to be true for $k$. We proceed to show them to be true for $k+1$. To that end, fix $x\in(k,k+1]$, and write
\begin{eqnarray}
&&\frac{P\left(\sum_{j=1}^{k+1}X_{nj}>xM_n\right)}{P(Y>M_n)^{k+1}}\nonumber\\
&=&\int_{(x-k,1]}\frac{P\left(\sum_{j=1}^kX_{nj}>(x-s)M_n\right)}{P(Y>M_n)^k}\frac{P(X_{n1}\in M_nds)}{P(Y>M_n)}\,.\label{l4.eq2}
\end{eqnarray}
Using the induction hypothesis \eqref{l4.claim} for $k$, and the fact that the function $s\mapsto\frac{P\left(\sum_{j=1}^kX_{nj}>(x-s)M_n\right)}{P(Y>M_n)^k}$ is non-decreasing, it follows that the convergence is uniform, that is,
$$
\lim_{n\to\infty}\sup_{s\in[x-k,1]}\left|\frac{P\left(\sum_{j=1}^kX_{nj}>(x-s)M_n\right)}{P(Y>M_n)^k}-c_k(x-s)\right|=0\,.
$$
Thus, as $n\to\infty$,
$$
\int_{(x-k,1]}\frac{P\left(\sum_{j=1}^kX_{nj}>(x-s)M_n\right)}{P(Y>M_n)^k}\frac{P(X_{n1}\in M_nds)}{P(Y>M_n)}
$$
\begin{equation}\label{l4.eq1}
=\int_{(x-k,1]}c_k(x-s)\frac{P(X_{n1}\in M_nds)}{P(Y>M_n)}+o(1)\,.
\end{equation}
Clearly,
$$
\lim_{n\to\infty}\int_{(x-k,1]}c_k(x-s)\frac{P(X_{n1}\in M_nds)}{P(Y>M_n)}=\int_{(x-k,1]}c_k(x-s)\alpha s^{-\alpha-1}ds=c_{k+1}(x)\,.
$$
This, in view of \eqref{l4.eq2} and \eqref{l4.eq1}, shows that \eqref{l4.claim} holds for $k+1$. That \eqref{l4.eq3} holds with $k+1$, follows trivially from the hypothesis that the same holds with $k$, and the definition of $c_{k+1}(\cdot)$. This completes the induction step, and thus completes the proof of the lemma.
\end{proof}

\begin{lemma}\label{l5}
If $k$ is a positive integer, and $(\vep_n)$ is a sequence of positive numbers such that
\begin{equation}\label{l5.eq1}
\vep_n=o\left(P(Y>M_n)^{1/k}\right)\,,
\end{equation}
then
\begin{equation}\label{l5.claim}
P\left(\sum_{j=1}^{k+1}X_{nj}>(k-\vep_n)M_n\right)\sim c_{k+1}(k)P(Y>M_n)^{k+1}\,,
\end{equation}
as $n\to\infty$.
\end{lemma}

\begin{proof}%[Proof of Lemma \ref{l5}]
The proof is again by induction on $k$. Our induction hypothesis is a bit stronger than the statement of the result, namely the following. For all $k\ge1$,
\begin{eqnarray}\label{l5.eq2}
&&P\left(\sum_{j=1}^{k+1}X_{nj}>(k-u)M_n\right)\\
&=&O\left(u^kP(Y>M_n)^k+P(Y>M_n)^{k+1}\right)\,,\nonumber
\end{eqnarray}
as $u\downarrow0$ and $n\to\infty$. In addition, if $(\vep_n)$ is a sequence of positive numbers satisfying \eqref{l5.eq1}, then \eqref{l5.claim} holds.
We first verify this hypothesis for $k=1$. To that end, fix $0<u<\delta<1/2$, and notice that
\begin{eqnarray*}
&&P\left(X_{n1}+X_{n2}>(1-u)M_n\right)\\
&=&\int_{(-u,1]}P\left(X_{n1}>(1-u-z)M_n\right)P\left(X_{n1}\in M_ndz\right)\\
&=&\int_{(-u,0]}+\int_{(0,\delta]}+\int_{(\delta,1-\delta]}+\int_{(1-\delta,1-u]}+\int_{(1-u,1]}\\
&=:&I_1+I_2+I_3+I_4+I_5\,.
\end{eqnarray*}
Clearly,
\begin{eqnarray*}
I_1&\le&P\left(X_{n1}>(1-u)M_n\right)\\
&=&O\left(uP(Y>M_n)\right)\,,
\end{eqnarray*}
as $u\downarrow0$ and $n\to\infty$, the last step following from Lemma \ref{l2}. Using that result once again, it follows that there is $C<\infty$ independent of $n$, $u$ and $\delta$, satisfying
\begin{eqnarray*}
I_2&\le&CP(Y>M_n)\int_{(0,\delta]}(u+z)P\left(X_{n1}\in M_ndz\right)\\
&\le&CP(Y>M_n)\left[u+M_n^{-1}\int_{(0,\delta M_n]}zP(Y\in dz)\right]\,.
\end{eqnarray*}
By the Karamata's theorem, it follows that as $n\to\infty$,
$$
M_n^{-1}\int_{(0,\delta M_n]}zP(Y\in dz)\sim\frac\alpha{1-\alpha}\delta^{1-\alpha}P(Y>M_n)\,.
$$
This shows that
$$
I_2=O\left[uP(Y>M_n)+P(Y>M_n)^2\right]\,.
$$
A restatement of Lemma \ref{l4} is that
$$
P(Y>M_n)^{-1}P\left(X_{n1}>\cdot\right)\to c_1(\cdot)\,,
$$
uniformly on $[\delta/2,1]$,
from which it follows that for $u\le\delta/2$,
$$
\lim_{n\to\infty}\sup_{\delta\le z\le1-\delta}\left|\frac{P\left(X_{n1}>(1-u-z)M_n\right)}{P(Y>M_n)}-c_1(1-u-z)\right|=0\,.
$$
Thus,
\begin{eqnarray*}
I_3&\sim&P(Y>M_n)\int_{(\delta,1-\delta]}c_1(1-u-z)P(X_{n1}\in M_ndz)\\
&\sim&P(Y>M_n)^{2}\alpha\int_\delta^{1-\delta}c_1(1-u-z)z^{-\alpha-1}dz\,,
\end{eqnarray*}
as $n\to\infty$. In other words,
$$
I_3=O\left(P(Y>M_n)^2\right)\,.
$$
By arguments similar to those leading to \eqref{l2.eq4}, it follows that for $\delta$ small enough and $n$ large enough,
\begin{eqnarray*}
I_4&\le&2\alpha {P(Y>M_n)}\int_{1-\delta}^{1-u}P\left(X_{n1}>(1-u-z)M_n\right)dz\\
&=&2\alpha {P(Y>M_n)}M_n^{-1}\int_0^{(\delta-u)M_n}P(Y>v)dv\\
&\le&2\alpha {P(Y>M_n)}M_n^{-1}\int_0^{\delta M_n}P(Y>v)dv\\
&\sim&2\frac\alpha{1-\alpha}\delta^{1-\alpha}P(Y>M_n)^2\,,
\end{eqnarray*}
as $n\to\infty$, the last step following from Karamata's theorem. Consequently,
\begin{equation*}
I_4=O\left(P(Y>M_n)^2\right)\,.
\end{equation*}
Finally, the same arguments show that,
\begin{eqnarray*}
I_5&=&O\left({P(Y>M_n)}\int_{1-u}^1P\left(X_{n1}>(1-u-z)M_n\right)dz\right)\nonumber\\
&=&O\left(uP(Y>M_n)\right)\,.
\end{eqnarray*}
The above calculations show that \eqref{l5.eq2} holds for $k=1$. Now suppose that $(\vep_n)$ is a sequence of positive numbers satisfying
$$
\vep_n=o\left(P(Y>M_n)\right)\,.
$$
Define $I_1$ to $I_5$ as above, with $u$ replaced by $\vep_n$. The same calculations as above, will reveal that
$$
I_3\sim\alpha\int_\delta^{1-\delta}c_1(1-z)z^{-\alpha-1}dz
$$
as $n\to\infty$, and
$$
\lim_{\delta\downarrow0}\limsup_{n\to\infty}\left[I_1+I_2+I_4+I_5\right]=0\,.
$$
Thus, \eqref{l5.claim} holds, again with $k=1$.

Next, we proceed to prove the induction step, that is, assuming that the hypothesis is true for $k-1$, we shall show that to be true for $k$.
Let $0<u<\delta<1/2$, and write
\begin{eqnarray*}
&&P\left(\sum_{j=1}^{k+1}X_{nj}>(k-u)M_n\right)\\
&=&\int_{(-u,1]}P\left(\sum_{j=1}^kX_{nj}>(k-u-z)M_n\right)P(X_{n1}\in M_ndz)\\
&=&\int_{(-u,0]}+\int_{(0,\delta]}+\int_{(\delta,1-\delta]}+\int_{(1-\delta,1-u]}+\int_{(1-u,1]}\\
&=:&J_1+J_2+J_3+J_4+J_5\,.\end{eqnarray*}
By similar arguments as above, it follows that
$$
J_1+J_2=O\left(u^kP(Y>M_n)^k+P(Y>M_n)^{k+1}\right)\,,
$$
as $u\downarrow0$ as $n\to\infty$, and
$$
J_3\sim P(Y>M_n)^{k+1}\int_{\delta}^{1-\delta}c_k(k-u-z)\alpha z^{-\alpha-1}dz\,,
$$
as $n\to\infty$.
Once again, by arguments similar to those leading to \eqref{l2.eq4}, it follows that for $\delta$ small enough and $n$ large enough,
\begin{eqnarray*}
J_4&\le&2\alpha P(Y>M_n)\int_{1-\delta}^{1-u}P\left(\sum_{j=1}^kX_{nj}>(k-u-z)M_n\right)dz\\
&\le&2\alpha\delta P(Y>M_n)P\left(\sum_{j=1}^{k}X_{nj}>(k-1)M_n\right)\,.
\end{eqnarray*}
Using the induction hypothesis \eqref{l5.eq2} for $k-1$, it follows that
$$
P\left(\sum_{j=1}^{k}X_{nj}>(k-1)M_n\right)=O\left(P(Y>M_n)^{k}\right)\,,
$$
which in turn, shows that there exists $C<\infty$ (independent of $\delta$, $u$ and $n$) satisfying
$$
J_4\le C\delta P(Y>M_n)^{k+1}\,,
$$
for $n$ large and $\delta$ small enough.
Finally, the same arguments show that,
\begin{eqnarray*}
J_5&=&O\left({P(Y>M_n)}\int_{1-u}^1P\left(\sum_{j=1}^kX_{nj}>(k-u-z)M_n\right)dz\right)\,.
\end{eqnarray*}
Using the induction hypothesis, we claim that there is $C<\infty$ such that
$$
P\left(\sum_{j=1}^kX_{nj}>(k-u-z)M_n\right)\le 
$$
$$
C\left[(u+z-1)^{k-1}P(Y>M_n)^{k-1}+P(Y>M_n)^k\right]\,,
$$
whenever $1-u\le z\le1$, and $u$ small enough. Thus, for such an $u$,
\begin{eqnarray*}
&&\int_{1-u}^1P\left(\sum_{j=1}^kX_{nj}>(k-u-z)M_n\right)dz\\
&\le&C\left[P(Y>M_n)^{k-1}\int_{1-u}^1(u+z-1)^{k-1}dz+P(Y>M_n)^k\right]\\
&=&C\left[P(Y>M_n)^{k-1}\frac{u^{k}}{k}+P(Y>M_n)^k\right]\,.
\end{eqnarray*}
This shows that
$$
J_5=O\left(u^{k}P(Y>M_n))^k+P(Y>M_n))^{k+1}\right)\,,
$$
and thus concludes the proof of \eqref{l5.eq2} for $k$. Once again, a close inspection of the calculations above will reveal that if $(\vep_n)$ is a sequence of positive numbers satisfying \eqref{l5.eq1}, then \eqref{l5.claim} holds. This proves the induction step, and thus completes the proof of the lemma
\end{proof}

\begin{proof}[Proof of Theorem \ref{t2}]
In view of Lemma \ref{l4}, \eqref{t2.claim2} is what remains to show. To that end, we start with the proof of the upper bound, that is the lim sup of the left hand side divided by the right one is at most 1. Define
$$
u_n:=P(Y>M_n)^{1/k}M_n,\,n\ge1\,.
$$
By \eqref{t2.eq1}, it follows that
$$
P(Y>M_n)^{-\frac{k+1}{k+2}}P(Y>u_n)\ll n^{-\frac1{k+2}}\,,
$$
which we restate as
$$
P(Y>u_n)\ll n^{-\frac1{k+2}}P(Y>M_n)^{\frac{k+1}{k+2}}\,.
$$
Clearly, the right hand side goes to zero as $n\to\infty$, and hence there exists a sequence $(z_n)$ satisfying
\begin{equation}\label{t2.eq2}
P(Y>u_n)\ll P(Y>z_n)\ll n^{-\frac1{k+2}}P(Y>M_n)^{\frac{k+1}{k+2}}\,.
\end{equation}
Fix such a $(z_n)$. An immediate consequence of the left inequality above is that
$$
z_n\ll P(Y>M_n)^{1/k}M_n\,.
$$
Fix a sequence $(\vep_n)$ satisfying
\begin{equation}\label{t2.eq3}
\frac{z_n}{M_n}\ll \vep_n\ll P(Y>M_n)^{1/k}\,.
\end{equation}

Define the events
\begin{eqnarray*}
E_n&:=&\left\{\sum_{i=1}^{k+1}X_{nj_i}>(k-\vep_n)M_n\mbox{ for some }1\le j_1\le\ldots\le j_{k+1}\le n\right\}\,,\\
F_n&:=&\left\{|X_{nj}|>z_n\mbox{ for at least }(k+2)\mbox{ many }j\mbox{'s }\le n\right\}\,,\\
G_n&:=&\left\{\sum_{j=1}^nX_{nj}\one\left(|X_{nj}|\le z_n\right)>\vep_nM_n\right\}\,,\\
H_n&:=&\left\{\sum_{i=1}^{k}X_{nj_i}>(k-\vep_n)M_n\mbox{ for some }1\le j_1\le\ldots\le j_k\le n\right\}\,.
\end{eqnarray*}
Clearly, for $n$ large enough so that $\vep_n\le1$, it holds that
$$
\{S_n>kM_n\}\subset E_n\cup F_n\cup G_n\cup H_n\,.
$$
Note that
\begin{eqnarray*}
P(E_n)&\le&\frac{n^{k+1}}{(k+1)!}P\left(\sum_{j=1}^{k+1}X_{nj}>(k-\vep_n)M_n\right)\\
&\sim&\frac{n^{k+1}}{(k+1)!}c_{k+1}(k)P(Y>M_n)^{k+1}\,,
\end{eqnarray*}
as $n\to\infty$, the equivalence following from Lemma \ref{l5} and the right inequality in \eqref{t2.eq3}. In order to complete the proof of the upper bound, we shall next show that
\begin{equation}\label{t2.eq4}
P(F_n)+P(G_n)+P(H_n)=o\left(n^{k+1}P(Y>M_n)^{k+1}\right)\,.
\end{equation}
To that end, notice that the right inequality in \eqref{t2.eq2} implies that
$$
P(F_n)=o\left(n^{k+1}P(Y>M_n)^{k+1}\right)\,.
$$
By Lemma \ref{l2}, it follows that
\begin{eqnarray*}
P(H_n)&=&O\left(n^k \vep_n^kP(Y>M_n)^k\right)\\
&=&o\left(n^k P(Y>M_n)^{k+1}\right)\\
&=&o\left(n^{k+1}P(Y>M_n)^{k+1}\right)\,,
\end{eqnarray*}
the equality in the second line following from the right inequality in \eqref{t2.eq3}. For estimating $P(G_n)$, we shall appeal once more to the result in \cite{prokhorov:1959}; see \eqref{l1.prokhorov}. First notice that
\begin{eqnarray*}
E\left|\sum_{j=1}^nX_{nj}\one\left(|X_{nj}|\le z_n\right)\right|&\le&n\int_{[0,{z_n}]}xP(|Y|\in dx)\\
&=&O\left(nz_nP(Y>z_n)\right)\\
&=&o\left[z_n\left\{nP(Y>M_n)\right\}^{\frac{k+1}{k+2}}\right]\\
&=&o(z_n)\\
&=&o(\vep_nM_n)\,,
\end{eqnarray*}
the last three steps following from the right inequality in \eqref{t2.eq2}, the limit in \eqref{eq.observe} and the left inequality in \eqref{t2.eq3} respectively. Thus, for $n$ large enough, \eqref{l1.prokhorov} implies that
$$
P(G_n)\le\exp\left\{-K_1\frac{\vep_nM_n}{z_n}\sinh^{-1}K_2\frac{\vep_nM_nz_n}{n\Var[X_{n1}\one(|X_{n1}|\le z_n)]}\right\}\,,
$$
for some finite and positive constants $K_1$ and $K_2$. Note that
\begin{eqnarray*}
\frac{\vep_nM_nz_n}{n\Var[X_{n1}\one(|X_{n1}|\le z_n)]}&\ge&\frac{\vep_nM_nz_n}{nE[X_{n1}^2\one(|X_{n1}|\le z_n)]}\\
&\sim&K\frac{\vep_nM_n}{nz_nP(Y>z_n)}\\
&\gg&\left[nP(Y>z_n)\right]^{-1}\\
&\gg&\left[nP(Y>M_n)\right]^{-(k+1)/(k+2)}\,,
\end{eqnarray*}
where $K$ is the constant from Karamata's theorem. In view of this, the fact that $\sinh^{-1}x\ge\log x$ for $x>0$, and \eqref{t2.eq3}, it follows that
$$
-K_1\frac{\vep_nM_n}{z_n}\sinh^{-1}K_2\frac{\vep_nM_nz_n}{n\Var[X_{n1}\one(|X_{n1}|\le z_n)]}\gg-\log\left[nP(Y>M_n)\right]\,,
$$
which shows that
$$
P(G_n)=o\left(n^{k+1}P(Y>M_n)^{k+1}\right)\,.
$$
This completes the proof of \eqref{t2.eq4}, and thus shows that
\begin{equation}\label{t2.eq5}
\limsup_{n\to\infty}\frac{P(S_n>kM_n)}{n^{k+1}P(Y>M_n)^{k+1}}\le\frac{c_{k+1}(k)}{(k+1)!}\,.
\end{equation}

For the reverse inequality with lim inf, notice that for all $\vep>0$,
$$
\{S_n>kM_n\}\supset\bigcup_{j\in C_{nk}}\left\{\sum_{i=1}^{k+1}X_{nj_i}>(k+\vep)M_n,\left|\sum_{i\in j^c}X_{ni}\right|\le\vep M_n\right\}\,,
$$
where $C_{nk}$ and $j^c$ are as defined just before Lemma \ref{l1}. By arguments similar to those proving \eqref{t3.eq15}, it follows that
\begin{eqnarray*}
&&P\left(\bigcup_{j\in C_{nk}}\left\{\sum_{i=1}^{k+1}X_{nj_i}>(k+\vep)M_n,\left|\sum_{i\in j^c}X_{ni}\right|\le\vep M_n\right\}\right)\\
&\sim&\sum_{j\in C_{nk}}P\left(\sum_{i=1}^{k+1}X_{nj_i}>(k+\vep)M_n,\left|\sum_{i\in j^c}X_{ni}\right|\le\vep M_n\right)\\
&=&\left(\begin{array}cn\\k+1\end{array}\right)P\left(\sum_{i=1}^{k+1}X_{ni}>(k+\vep)M_n\right)P\left(\left|\sum_{i=k+2}^nX_{ni}\right|\le\vep M_n\right)\,.
\end{eqnarray*}
By Lemma \ref{l4} and the fact that $M_n^{-1}\sum_{i=k+2}^nX_{ni}$ goes to zero in probability, it follows that the right hand side is asymptotically equivalent to
$$
\frac{n^{k+1}}{(k+1)!}P(Y>M_n)^{k+1}c_{k+1}(k+\vep)\,.
$$
The above calculations put together show that
$$
\liminf_{n\to\infty}\frac{P(S_n>kM_n)}{n^{k+1}P(Y>M_n)^{k+1}}\ge\frac{c_{k+1}(k+\vep)}{(k+1)!}\,.
$$
By letting $\vep\downarrow0$, the reverse inequality of \eqref{t2.eq5} with lim inf, follows. This completes the proof.
\end{proof}

\section{Examples}\label{sec:examples}
We end the paper with a couple of examples. The first example is of a random variable with a regularly varying tail, for which \eqref{assume} does not hold. Let $\alpha>0$ and suppose $Y$ is a random variable whose tail probability is given by
$$
P(Y>x)=\frac12\left(1+\frac1{\lfloor x\rfloor}\right)x^{-\alpha},\,x\ge1\,.
$$
Set
$$
a_n:=1-n^{-2},\,n\ge1\,,
$$
and observe that
\begin{eqnarray*}
\lim_{n\to\infty}\frac1{1-a_n}\left[\frac{P(Y>na_n)}{P(Y>n)}-a_n^{-\alpha}\right]=1\,.
\end{eqnarray*}
Clearly, \eqref{assume} cannot hold for this $Y$, although the tail of $Y$ is regularly varying with index $-\alpha$.

Suppose that $P(Y>\cdot)$ is regularly varying with index $-\alpha$. By Karamata's representation, there exist functions $c(\cdot)$ and $\vep(\cdot)$ such that
\begin{eqnarray}
\nonumber P(Y>x)&=&c(x)x^{-\alpha}\exp\left(\int_1^x\frac{\vep(t)}tdt\right),\,x\ge1\,,\\
\label{neweq3}\lim_{x\to\infty}c(x)&=&c\in(0,\infty)\,,\\
\nonumber \text{and }\lim_{x\to\infty}\vep(x)&=&0\,.
\end{eqnarray}
In the following theorem, we show that if \eqref{neweq3} is strengthened to $c(x)=c$ for $x$ large enough, then \eqref{assume} holds.

\begin{theorem}
Assume that $Y$ is a random variable such that there exists $\alpha,c,T>0$ and $\vep:[0,\infty)\to\bbr$ satisfying 
\[
\lim_{x\to\infty}\vep(x)=0\,,
\]
and
\begin{equation}\label{eq.normrv}
P(Y>x)=cx^{-\alpha}\exp\left(\int_1^x\frac{\vep(t)}tdt\right)\,,\mbox{ for all }x>T\,.
\end{equation}
Then, $Y$ satisfies \eqref{assume}.
\end{theorem}

\begin{proof}
Fix $\delta\in(0,1)$, and let $T_0\ge T$ be such that for all $x\ge T_0$, it holds that $|\vep(x)|\le\delta$. Suppose that $\frac12\le a\le1$ and $x\ge2T_0$. Then,
\begin{eqnarray}
\nonumber\left|\frac{P(Y>ax)}{P(Y>x)}-a^{-\alpha}\right|
\nonumber&=&a^{-\alpha}\left|\exp\left(-\int_{ax}^x\frac{\vep(t)}tdt\right)-1\right|\\
\nonumber&\le&a^{-\alpha}\int_{ax}^x\frac{|\vep(t)|}tdt\,\exp\left(\int_{ax}^x\frac{|\vep(t)|}tdt\right)\\
\nonumber&\le&\delta a^{-\alpha-1}\log(1/a)\\
\label{ineq1}&\le&2^{\alpha+1}\delta(1-a)\,.
\end{eqnarray}
Now suppose that $\frac12\le a\le b\le1$ and $x\ge2T_0$, and note that 
\begin{eqnarray}
\nonumber&&\left|\frac{P(Y>ax)-P(Y>bx)}{P(Y>x)}-(a^{-\alpha}-b^{-\alpha})\right|\\
\label{eq.Q1}&\le&\frac{P(Y>bx)}{P(Y>x)}\left|\frac{P(Y>ax)-P(Y>bx)}{P(Y>bx)}-\left(\frac ab\right)^{-\alpha}+1\right|\\
\label{eq.Q2}&&+\left|\frac{P(Y>bx)}{P(Y>x)}\left[\left(\frac ab\right)^{-\alpha}-1\right]-(a^{-\alpha}-b^{-\alpha})\right|\,.
\end{eqnarray}
The quantity in \eqref{eq.Q1} is clearly at most
\[
\frac{P(Y>x/2)}{P(Y>x)}\left|\frac{P(Y>ax)}{P(Y>bx)}-\left(\frac ab\right)^{-\alpha}\right|\,.
\]
Using \eqref{ineq1}, first with $a=1/2$, and then with $x$ and $a$ replaced by $bx$ and $a/b$, respectively, it follows that
the above quantity is at most $c(\alpha)\delta(b-a)$, where $c(\alpha)$ is a finite constant depending only on $\alpha$. The quantity in \eqref{eq.Q2} equals
\begin{eqnarray*}
\left[\left(\frac ab\right)^{-\alpha}-1\right]\left|\frac{P(Y>bx)}{P(Y>x)}-b^{-\alpha}\right|&\le&2^{\alpha+1}\delta(1-b)\left[\left(\frac ab\right)^{-\alpha}-1\right]\\
&\le&2^{\alpha+1}(1-b)\alpha\left(\frac ab\right)^{-\alpha-1}\left(1-\frac ab\right)\,,
\end{eqnarray*}
the first inequality following, once again,  from \eqref{ineq1}. Thus, there exists a finite constant $C(\alpha)$, depending only on $\alpha$, such that 
\[
\left[\left(\frac ab\right)^{-\alpha}-1\right]\left|\frac{P(Y>bx)}{P(Y>x)}-b^{-\alpha}\right|\le C(\alpha)\delta(b-a)\,.
\]
This completes the proof.
\end{proof}

It is known that \eqref{eq.normrv} is equivalent to assuming that $h(\cdot)$, defined by
\[
h(x):=x^\alpha P(Y>x),\,x\ge0\,,
\]
belongs to the Zygmund class, that is, for all $\vep>0$, $x^{-\vep}h(x)$ is eventually non-increasing, and $x^{\vep}h(x)$ is eventually non-decreasing. Note that the Zygmund class is a subclass of the slowly varying functions. 
See Chapter 1.5.3 of \cite{bingham:goldie:teugels:1987} for a proof of these facts and other details. For example, if $Y$ is a random variable with c.d.f.
$$
F(x):=\left\{\begin{array}{ll}\max\left\{1/2,1-x^{-\alpha}(\log x)^{-2}\right\},&x\ge0\,,\\\min\left\{1/2,|x|^{-\alpha}(\log |x|)^{-2}\right\},&x<0\,,\end{array}\right.
$$
then $Y$ satisfies the hypotheses of Theorems \ref{t1} and \ref{t2},  when $\alpha\ge k$ and $\alpha<k/(k+2)$ respectively.

\section*{Acknowledgement} The author takes this opportunity to gratefully acknowledge the contribution of Professor B.V.\ Rao in his professional life. Professor Rao introduced him to the wonderful world of probability through the several courses in the B.\ Stat and M.\ Stat program of  the Indian Statistical Institute.

The author's research is partially supported by the INSPIRE grant of the Department of Science and Technology, Government of India. He also thanks an anonymous referee for comments and suggestions that helped to improve the exposition significantly.

\end{document}